\documentclass[leqno,10pt]{amsart}
\usepackage{amsfonts,amsmath,amssymb}
\usepackage[all]{xy}
\usepackage[dvips]{graphicx}

\newcommand{\co}{\colon\thinspace}

\newcommand{\TC}{\mathbf{TC}}

\newcommand{\wgt}{\mathrm{wgt}}

\newcommand{\cat}{\mathrm{cat}}
\newcommand{\gen}{\mathfrak{genus}}
\newtheorem{thm}{Theorem}[section]

\newtheorem{prop}[thm]{Proposition}
\newtheorem{defn}[thm]{Definition}

\newtheorem{exam}[thm]{Example}

\newtheorem{rem}[thm]{Remark}
\newtheorem{questions}[thm]{Questions}
\newtheorem{problem}[thm]{Problem}

\begin{document}

\keywords{Topological complexity, Schwarz genus, weights of cohomology classes, Massey products}
\subjclass[2000]{Primary 55M99, 55S30; Secondary 68T40.}
\thanks{The author was supported by a grant from the UK EPSRC}

\title[Topological complexity]{Topological complexity of motion planning\\
and Massey products}

\author[M. ~Grant]{Mark Grant}
\address{Department of Mathematical Sciences, Durham University\\
South Road, Durham, DH1 3LE, UK}
\email{mark.grant@durham.ac.uk}

\maketitle

\begin{abstract} We employ Massey products to find sharper lower
  bounds for the Schwarz genus of a fibration than those previously known. In particular we give examples of non-formal spaces
$X$ for which the topological complexity $\TC(X)$ (defined to be the genus of the free path fibration on $X$) is greater than the
zero-divisors cup-length plus one.
\end{abstract}

\section{Introduction.}

Motion planning is a fundamental area of research in
Robotics. A {\em motion planning algorithm} for a given mechanical
system $S$ is a function which assigns to each ordered pair $(A,B)$ of
physical states of $S$ a continuous motion of $S$ starting at $A$ and ending at
$B$. We may regard the admissible physical states of $S$ as being parameterised
by the points of a topological space $X$ (the {\em configuration
space} of the system) such that
motions of the system correspond to continuous paths $\gamma\co
[0,1]=I\to X$. A motion planning algorithm for the system is then a
section $s\co X\times X\to X^I$ (not necessarily continuous) of the free path fibration
\begin{equation}
\pi_X\co X^I\to X\times X,\quad \pi_X(\gamma)=(\gamma(0),\gamma(1)).
\end{equation}
The minimum number of domains of continuity of such a section $s$
provides a measure of the complexity of
the motion planning problem in $X$. This observation led M.\ Farber in
[Far1], [Far2] to consider a new numerical homotopy
invariant, called the {\em topological
  complexity} of the configuration space $X$ and denoted $\TC(X)$,
which may be defined to be
the Schwarz genus ([Sch], see Section 2) of the fibration
(1). The invariant $\TC(X)$ is a close relative of the
Lusternik-Schnirelmann category $\cat(X)$, and although independent the two
satisfy the inequalities $\cat(X)\leq\TC(X)\leq\cat (X\times
X)\leq 2\cdot\cat(X)-1$. We refer
the reader to [Far3] for an excellent survey of results
in this area.

 Computing $\TC(X)$ for a given $X$ can be an extremely difficult
 task (for example, by the main result of [FTY] the topological complexity of real projective space
 $\TC (\mathbf{R}P^n)$ for $n\neq 1,3,7$ equals one plus the smallest dimension of
 Euclidean space into which $\mathbf{R}P^n$ immerses). As in the case
 of LS-category one applies cohomology theory to find computable lower bounds. One such lower bound for $\TC(X)$, which only requires knowledge of the cohomology algebra of $X$, is given in [Far1]. If $X$ is a space of finite type
 and $\mathbf{F}$ is a field, there is an isomorphism of graded algebras $H^*(X\times X;\mathbf{F})\cong H^*(X;\mathbf{F})\otimes H^*(X,\mathbf{F})$,
where the product on the right is given by
$$(\alpha\otimes \beta)(\gamma\otimes
\delta)=(-1)^{|\beta||\gamma|}\alpha\gamma\otimes\beta\delta.$$
The cup product map $\cup\co H^*(X;\mathbf{F})\otimes
H^*(X;\mathbf{F})\to H^*(X;\mathbf{F})$ is a ring homomorphism, whose
kernel is the ideal of {\em zero-divisors}. The {\em zero-divisors
  cup-length over $\mathbf{F}$} is the number of factors in the
longest non-trivial product of zero-divisors. Then $\TC(X)$ is greater
than the zero-divisors
  cup-length over $\mathbf{F}$, for
any field of coefficients $\mathbf{F}$.

 In a recent paper of Farber and the author [FG2], stable cohomology operations are utilised to obtain sharper lower bounds for $\TC$ than the
 zero-divisors cup-length. In this article we investigate the effects
 of Massey products on topological complexity. The key notion is that of {\em weight of a cohomology
 class with respect to a fibration}, first defined in [FG1], which
generalises the {\em category weight} of Y.\ Rudyak [Rud] and J.\
Strom [Str] (which in turn are
 refinements of the original notion of category weight due to E.\ Fadell and S.\ F.\ Husseini
 [FH]). In Section 2 we recall some properties of this weight, and show how
 classes of high category weight can lead to classes of high
 weight with respect to the free path fibration. In Section 3 we briefly
 review Massey products, and show how they may be
 used to estimate the Schwarz genus of a fibration, generalising a
 result of Rudyak ([Rud], Theorem 4.4). In the final Section 4
 we give examples of non-formal spaces where non-zero Massey products can be employed to find better lower bounds for $\TC$ than the zero-divisors cup-length.

The author wishes to thank Michael Farber, Thomas Kahl and Sergey Yuzvinsky for stimulating discussions regarding this work.

\section{Weights of cohomology classes with respect to a fibration.}
 In this Section we recall the definition of {\em weight of a
   cohomology class with respect to a fibration} from [FG1]. We also give an alternative
 characterisation of weight in terms of fibred joins (Proposition \ref{join}),
 and show that classes with high category weight may lead to classes
 with high weight with respect to the path fibration $\pi_X$ (Theorem
 \ref{main}). {\em In this article, all spaces are assumed to be
   path-connected and of finite type. Unless specified otherwise, coefficients for
   cohomology are taken in an arbitrary commutative ring $R$ with unit.}

Let $p\co E\to B$ be a fibration. The {\em Schwarz genus} of $p$,
denoted $\gen(p)$, is defined to be the minimum $k$ such that $B$
may be covered by open subsets $U_1,\ldots ,U_k$, on each of which
$p$ admits a continuous local section (a map $s_i\co U_i\to E$ such that $p\circ
s_i$ is the identity map on $U_i$).

The concept of genus was defined and thoroughly studied by A.\ S.\ Schwarz
[Sch]; it is also called {\em sectional category} in the modern
literature. It generalises the Lusternik-Schnirelmann category, in the
following sense. Let $(X,x_0)$ be a pointed space and
let $p_X\co P_0X\to X$ be the Serre path fibration on $X$, where
$P_0X=\{\gamma\co I\to X\mid\gamma(0)=x_0\}$ and
$p_X(\gamma)=\gamma(1)$. Then $\gen(p_X)=\cat(X)$ (here we do not
normalise, so that for us $\cat(Y)=1$ if $Y$ is contractible). Other notable
applications of the genus include the works of S.\ Smale [Sma]
and V.\ A.\ Vassilliev [Vas1], [Vas2] on the complexity of
algorithms for finding roots of polynomial equations, and applications
to the embedding problem for topological manifolds (see Chapter VII of [Sch] and the references therein).

Another important application of the genus (and the one with which we are most concerned here) is to the motion planning problem in Robotics. For any space $X$ let $X^I$ denote the space of paths in $X$ (with no restrictions on end-points) with the compact-open topology. The {\em topological complexity} of $X$ is defined by $\TC(X)=\gen(\pi_X)$, where
$$\pi_X\co X^I\to X\times X,\quad\pi_X(\gamma)=(\gamma(0),\gamma(1))$$
is the free path fibration. As mentioned in the Introduction, the number $\TC(X)$ provides a measure of the complexity of the motion planning problem for a system with configuration space homotopy equivalent to $X$. More details can be found in Farber [Far1], [Far2], [Far3].

A useful cohomological lower bound for the genus of an arbitrary fibration $p\co E\to B$ was given by Schwarz.
\begin{thm}[Schwarz, \cite{[Sch]}, Theorem 4]\label{Schwarz}
Suppose there are classes $u_1,\ldots,u_\ell\in H^*(B)$ such that $p^*(u_i)=0$ for $i=1,\ldots, \ell$ and the product $u_1\cdots u_\ell$ is non-zero. Then $\gen(p)>\ell$.
\end{thm}

In this Theorem one may also use local coefficients, or other
cohomology theories, but we will not do so here. Note that for the
Serre path fibration $p_X$ Theorem \ref{Schwarz} gives the classical lower bound
for $\cat(X)$ in terms of the cup-length of $\tilde{H}^*(X)$ (since
$P_0X$ is contractible). For the free path fibration $\pi_X\co X^I\to
X\times X$ and coefficients in a field, Theorem \ref{Schwarz} gives the lower
bound for $\TC(X)$ in terms of zero-divisors cup-length [Far1]
described in the Introduction. This is because $\pi_X$ is
homotopically equivalent to the diagonal map $\triangle\co X\to
X\times X$.

In [FH] it was observed by Fadell and Husseini that some
indecomposables in $\tilde{H}^*(X)$ carry more weight than others in
the cup-length estimate for $\cat(X)$ (an homogenous element $u$ in a
graded algebra is called {\em indecomposable} if it cannot be written
as a sum of products $u=\sum v_iw_i$ where the dimensions of each
$v_i$, $w_i$ are strictly less than that of $u$). Their definition of
{\em category weight} of a cohomology class was later refined by
Rudyak [Rud] and by Strom [Str]. The notion of weight was generalised to an arbitrary fibration $p\co E\to B$ in papers [FG1], [FG2].
\begin{defn}{\em 
The }weight {\em of a non-zero cohomology class $u\in H^*(B)$ with respect to $p$, denoted $\wgt_p(u)$, is defined by
\[\wgt_p(u)=\sup\{ k\mid f^*(u)=0\textrm{ for all maps }f\co A\to
X\textrm{ with }\gen(f^*p)\leq k\}.\]
Here $f^*p$ denotes the pull-back fibration of $p$ along $f$.}
\end{defn}
\begin{rem}{\em
The (strict) category weight of a class
  $u\in H^*(X)$ is defined in [Rud] to be
$$\wgt(u)=\sup\{ k\mid f^*(u)=0\textrm{ for all maps }f\co A\to
X\textrm{ with }\cat(f)\leq k\}$$
(recall that $\cat(f)$ is the smallest $n$ such that $A$ admits an open cover
$U_1,\ldots ,U_n$ with $f|_{U_i}$ null-homotopic for all $i$).
 It is not difficult to see that $\wgt(u)=\wgt_{p_X}(u)$, the weight
 of $u$ with respect to the Serre path fibration.}
\end{rem}
An alternative characterisation of weight may be given, in terms of fibred joins. Recall that the {\em $k$-fold iterated fibred join} of a fibration $p\co E\to B$ with fibre $F$ is a fibration $p(k)\co E(k)\to B$ with fibre $\ast^kF$, the $k$-fold join of $F$ with itself. The domain space $E(k)$ has underlying set the formal linear sums
\[\tilde{e}=e_1t_1+e_2t_2+\ldots +e_kt_k,\quad e_i\in E,\,t_i\in [0,1],\quad\sum t_i=1,\quad p(e_1)=\dots =p(e_k),
\]
with the understanding that two such sums $\tilde{e}$ and $\tilde{e'}$
are equal if and only if $e_i=e'_i$ whenever $t_i>0$. Its topology is
defined to be the smallest topology such that the co-ordinate maps
\[t_i\co E(k)\to [0,1],\quad e_i\co t_i^{-1}(0,1]\to E\]
are all continuous. The projection $p(k)\co E(k)\to B$ is defined by
\[p(k)(e_1t_1+\cdots +e_kt_k)=p(e_1)=\dots =p(e_k).\]
Note that $p(1)\co E(1)\to B$ is exactly $p\co E\to B$. Schwarz proved ([Sch], Theorem 3) that {\em $\gen(p)\leq k$ if and only if $p(k)$ has a section}.
\begin{prop}\label{join}
For any non-zero $u\in H^*(B)$ we have
\[\wgt_p(u)=\sup\{ k\mid p(k)^*(u)=0\}.\]
In particular, $\wgt_p(u)\geq 1$ if and only if $p^*(u)=0$.
\end{prop}

\begin{proof}
 As is shown in Proposition 34 of [FG1], if $p(k)^*(u)=0$ then
$\wgt_p(u)\geq k$. Hence $\wgt_p(u)\geq\sup\{ k\mid p(k)^*(u)=0\}.$

Now suppose that $\wgt_p(u)= k$, and consider the pull-back
fibration $p(k)^* p$. It has base space $E(k)$ and total space
\[\{(e_1t_1+\ldots +e_kt_k,e)\in E(k)\times E\mid p(e_1)=\dots =p(e_k)=p(e)\}.\]
The open sets $U_i=t_i^{-1}(0,1]$, $i=1,\ldots ,k$ cover $E(k)$, and
on each there is a section $s_i$ of $p(k)^* p$ given by
\[s_i(e_1t_1+\ldots +e_kt_k)=(e_1t_1+\ldots +e_kt_k,e_i).\] Hence
$\gen(p(k)^* p)\leq k$, and so $p(k)^*(u)=0$.
\end{proof}

\begin{thm}[\cite{[FG1]}, Theorem 33]\label{weight}
Suppose there are classes $u_1,\ldots,u_\ell\in H^*(B)$ whose product $u_1\cdots u_\ell$ is non-zero. Then
$$\gen(p)>\wgt_p(u_1\cdots u_\ell)\geq\sum_{i=1}^{\ell}\wgt_p(u_i).$$
\end{thm}

Theorem \ref{weight} may give a better lower bound for $\gen(p)$ than Theorem \ref{Schwarz},
provided one can find indecomposables $u\in H^*(B)$ with
$\wgt_p(u)>1$. Fadell and Husseini achieved this in the case of
category weight, using stable cohomology operations ([FH]
Theorem 3.12, see also Corollary 4.7 of [Rud]). An analogous result for $\TC$ was obtained by the authors in [FG2], where stable cohomology operations are used to find indecomposable zero-divisors $z\in H^*(X\times X)$ with $\wgt_{\pi_X}(z)>1$, thus allowing the computation of $\TC$ of various lens spaces. Rudyak has shown ([Rud], Corollary 4.6) that if $u\in H^*(X)$ is a Massey product then $\wgt(u)>1$ (the definition of Massey's triple product will be recalled in Section 3). To conclude this Section we show how classes of high category weight can lead to zero-divisors with high weight with respect to $\pi_X$.
\begin{thm}\label{main}
Let $X$ be an $r$-connected space, $r\geq 1$. Suppose that $u\in
H^\ell(X;\mathbf{F})$ has $\wgt(u)\geq k\geq 1$, where $k(r+1)\leq\ell
<(k+1)(r+1)$ and $\mathbf{F}$ is a field. Then there exists an element $\phi(u)\in H^\ell(X\times
X;\mathbf{F})$, of the form
\begin{equation}
\phi(u)=1\times u+\theta(u),\quad\theta(u)\in\bigoplus_{\stackrel{i+j=\ell}{i>0}}H^i(X;\mathbf{F})\otimes H^j(X;\mathbf{F}),
\end{equation}
which has $\wgt_{\pi_X}(\phi(u))\geq k$. If the cup products
$H^i(X;\mathbf{F})\otimes H^{\ell -i}(X;\mathbf{F})\to H^{\ell}(X;\mathbf{F})$
for $0<i<\ell$ all vanish, then
\[\phi(u)=\overline{u}=1\times u -u\times 1.\]
\end{thm}

\begin{proof}
 The $k$-fold fibred joins of the Serre fibration $p_X$ and the
free path fibration $\pi_X$ are related by the following diagram,
\begin{equation}\xymatrix{
\ast^k\Omega X \ar@{=}[r] \ar[d] & \ast^k\Omega X \ar[d]\\
P_0X(k) \ar[r] \ar[d]_{p_X(k)} &  X^I(k)\ar[d]^{\pi_X(k)} \\
X \ar[r]^{\iota} & X\times X,}
\end{equation}
where the bottom square is a pull-back and the map $\iota\co X\to X\times X$ is given by $\iota(x)=(x_0,x)$. Let $(E_r,d_r)$ and $(\bar{E}_r,\bar{d}_r)$ denote the Leray-Serre spectral sequences of $p_X(k)$ and $\pi_X(k)$ respectively. The class $u\in H^\ell(X)=E^{\ell,0}_2$ has $\wgt(u)\geq k$, and therefore by Proposition \ref{join} lies in the kernel of $p_X(k)^*\co H^\ell(X)\to H^\ell(P_0X(k))$, which is known to correspond to the edge homomorphism
\[
H^\ell(X)=E_2^{\ell,0}\twoheadrightarrow E_{\infty}^{\ell,0}\hookrightarrow H^\ell(P_0X(k))
\]
(see for example [Whi] p. 649).

Since $X$ is $r$-connected, the based loop space $\Omega X$ is $(r-1)$-connected; hence by Lemma 2.3 of [Mil] the common fibre $\ast^k\Omega X$ is $(rk+k-2)$-connected. For dimensional reasons $u$ must therefore be in the image of the differential
\[
d_\ell\co H^{\ell-1}(\ast^k\Omega X)=E^{0,\ell-1}_\ell\to E^{\ell, 0}_\ell=H^\ell(X).
\]
Let $v$ be in $H^{\ell-1}(\ast^k\Omega X)=E^{0,\ell-1}_\ell=\bar{E}^{0,\ell-1}_\ell$ with $d_\ell(v)=u$. We set
$$\phi(u)=\bar{d}_\ell(v)\in\bar{E}^{\ell,0}_\ell=\bar{E}^{\ell,0}_2 =H^\ell(X\times X).$$
By naturality of spectral sequences and using diagram (3) we see that
$$\iota^*(\phi(u))=\iota^*(\bar{d}_\ell(v))=d_\ell(v)=u,$$
and hence $\phi(u)$ is of the form (2). Since $\phi(u)$ is in the
image of the differential $\bar{d}_\ell$ it is in the kernel of the
edge homomorphism
\[
H^\ell(X\times X)=\bar{E}_2^{\ell,0}\twoheadrightarrow \bar{E}_{\infty}^{\ell,0}\hookrightarrow H^\ell(X^I(k)),
\]
which corresponds to $\pi_X(k)^*\co H^\ell(X\times X)\to
H^\ell(X^I(k))$. Hence $\wgt_{\pi_X}(\phi(u))\geq k$, proving the first
statement.

Now note that $\wgt_{\pi_X}(\phi(u))\geq k\geq 1$ implies that
$\phi(u)$ is a zero-divisor. Hence
$\triangle^*(\phi(u))=u+\triangle^*(\theta(u))=0$ where
$\triangle^*\co H^*(X)\otimes H^*(X)\to H^*(X)$ is the cup product
map, and the second statement follows.
\end{proof}

\section{Massey products.}
In this Section we recall some definitions and results concerning
Massey products and show how they may be used to estimate the Schwarz
genus of a fibration, generalising a result of Rudyak ([Rud],
Theorem 4.4). We consider only the triple product [UM], [Mas], which is a secondary cohomology operation of three
variables, but much of what we say may be generalised to higher order or
matric Massey products (see [Kra] and [May] for
definitions).

Let $X$ be a topological space. The singular cochain complex of $X$ with coefficients in $R$, denoted $C^*(X)$, is a DGA over $R$ with cochain multiplication
$\bullet\co C^*(X)\otimes C^*(X)\to C^*(X)$
defined in the usual way and differential $d$ of degree $+1$ satisfying $d(a\bullet b)=da\bullet b+(-1)^{|a|}a\bullet db$. Given cohomology classes $\alpha
,\beta ,\gamma\in H^*(X)$ of dimensions $p$, $q$ and $r$ such that
$\alpha\beta=0=\beta\gamma$, their {\em Massey product} is a subset $$\langle\alpha ,\beta ,\gamma
\rangle\subseteq H^{p+q+r-1}(X)$$ defined as follows. Let $a,b,c\in C^*(X)$ be cocycles representing $\alpha$, $\beta$ and $\gamma$
respectively. Since $\alpha\beta=0$ there is a cocycle $\mu\in
C^{p+q-1}(X)$ with $d\mu =a\bullet b$. Similarly, since $\beta\gamma=0$ there is a cocycle $\lambda\in
C^{q+r-1}(X)$ with $d\lambda =b\bullet c$. The cochain
$a\bullet\lambda+(-1)^{p+1}\mu\bullet c$ is a cocycle which therefore
represents a class in $H^{p+q+r-1}(X)$. The Massey product
$\langle\alpha,\beta,\gamma\rangle$ is the set of all cohomology classes arising
in this way,
$$\langle\alpha,\beta,\gamma\rangle=\{[a\bullet\lambda+(-1)^{p+1}\mu\bullet c]\in
H^{p+q+r-1}(X)\mid d\mu=a\bullet b\textrm{
  and }d\lambda=b\bullet c\}.$$
Elements of the above Massey product differ by elements of the subgroup
$$\alpha H^{q+r-1}(X)+H^{p+q-1}(X)\gamma\subseteq
H^{p+q+r-1}(X),$$
which is termed the {\em indeterminacy} of $\langle\alpha,\beta,\gamma\rangle$. Hence
one may regard $\langle\alpha,\beta,\gamma\rangle$ as an element of the quotient
group of $H^{p+q+r-1}(X)$ modulo this indeterminacy. Note that if all
cup products in $H^*(X)$ are zero the indeterminacy vanishes. We will
say that $\langle\alpha,\beta,\gamma\rangle$ is {\em non-zero} if $0\notin
\langle\alpha,\beta,\gamma\rangle$.

The next result is based on Theorem 4.4 of [Rud].
\begin{thm}\label{rudyak}
 Let $p\co E\to B$ be a fibration, and let
  $\alpha,\beta,\gamma\in H^*(B)$ be cohomology classes. If the Massey
  product $\langle\alpha,\beta,\gamma\rangle$ is defined and non-zero, then
$$\gen(p)>\wgt_p(\beta)+\min \{\wgt_p(\alpha),\wgt_p(\gamma)\}.$$
\end{thm}
\begin{proof}
Let $k=\wgt_p(\beta)$ and $\ell=\min
\{\wgt_p(\alpha),\wgt_p(\gamma)\}$. Assume that
$\langle\alpha,\beta,\gamma\rangle$ is defined, and that $\gen(p)\leq k+\ell$. This
means there exist open subsets $C_i$ for $i=1,\ldots ,k$ and $D_j$ for
$j=1,\ldots ,\ell$ of $B$ such that
$$C=\bigcup_{i=1}^{k}C_i,\quad D=\bigcup_{j=1}^{\ell}D_j,\qquad B=C\cup D,$$
and $p$ admits a local section on each $C_i$, $D_j$. From the definition of
weight it follows that $\beta |_C=0$ and $\alpha
|_D=0=\gamma |_D$. A cocycle $b$ which represents $\beta$ is
therefore the image of a cocycle $\tilde{b}\in
C^*(B,C)$ which vanishes on cycles in $C$, by the exact cohomology
sequence of the pair $(B,C)$. Similarly the cocycles $a$ and $c$
representing $\alpha$ and $\gamma$ are the images of cocycles
$\tilde{a},\tilde{c}\in C^*(B,D)$. A quick glance at the diagram
\begin{equation}\xymatrix{
C^*(B,C)\otimes C^*(B,D) \ar[r]^{\bullet} \ar[d] & C^*(B,C\cup D) \ar[d]\\
C^*(B)\otimes C^*(B) \ar[r]^{\bullet} & C^*(B)}
\end{equation}
given by naturality of cochain multiplication now shows that $a\bullet
b=0=b\bullet c$, since $C^*(B,C\cup D)=0$. It follows that the Massey
product $\langle\alpha,\beta,\gamma\rangle$ contains zero.
\end{proof}

In the next Section we will apply Theorem \ref{rudyak} to obtain lower bounds
for $\TC(X)$ sharper than the zero-divisors cup-length, for certain
spaces $X$. The next two Propositions gather some facts
about Massey products which are needed in the sequel.
\begin{prop}\label{multi}
{\bf (a) }(Linearity) If $\langle\alpha,\beta,\gamma \rangle$ and $\langle\alpha',\beta,\gamma \rangle$ are defined and $|\alpha|=|\alpha'|$, then $\langle\alpha+\alpha',\beta,\gamma\rangle$ is defined and $$\langle\alpha+\alpha',\beta,\gamma\rangle\subseteq \langle\alpha,\beta,\gamma\rangle+\langle\alpha',\beta,\gamma\rangle.$$
Similar statements hold in the second and third variables.

{\bf (b) }(Scalar multiplication) If $\langle\alpha,\beta,\gamma
\rangle$ is defined and $r\in R$, then $\langle
r\alpha,\beta,\gamma\rangle$ is defined and
$$r\langle\alpha,\beta,\gamma\rangle\subseteq \langle r\alpha,\beta,\gamma \rangle.$$
Similar statements hold in the second and third variables. If $u\in R$
is a unit, then
$$u\langle\alpha,\beta,\gamma\rangle=\langle
u\alpha,\beta,\gamma\rangle =\langle\alpha,u\beta,\gamma\rangle =\langle\alpha,\beta,u\gamma\rangle.$$

{\bf (c) }(Naturality) If $f\co Y\to X$ is a map, then
$$f^*\langle\alpha,\beta,\gamma\rangle\subseteq\langle f^*(\alpha),f^*(\beta),f^*(\gamma)\rangle.$$

{\bf (d) }(Internal products) If $\langle\alpha,\beta,\gamma\rangle$
is defined and $\alpha',\beta',\gamma'\in H^*(X)$ are arbitrary
cohomology classes, then
$\langle\alpha\alpha',\beta\beta',\gamma\gamma'\rangle$ is
defined. Furthermore, if the latter operation has zero indeterminacy
then
$$\langle\alpha,\beta,\gamma\rangle\alpha'\beta'\gamma'=\pm
\langle\alpha\alpha',\beta\beta'\gamma\gamma'\rangle.$$
(The similar relation
$\alpha\beta\gamma\langle\alpha',\beta',\gamma'\rangle=\pm\langle\alpha\alpha',\beta\beta'\gamma\gamma'\rangle$
holds when $\langle\alpha',\beta',\gamma'\rangle$ is defined and
$\alpha,\beta,\gamma$ are arbitrary.)

{\bf (e) }(External products) If $\langle\alpha_1,\beta_1,\gamma_1\rangle$ is defined in $H^*(X_1)$
and $\alpha_2,\beta_2,\gamma_2\in H^*(X_2)$ are arbitrary
cohomology classes, then
$\langle\alpha_1\times\alpha_2,\beta_1\times\beta_2,\gamma_1\times\gamma_2\rangle$
is defined in $H^*(X_1\times
X_2)$. Furthermore, if the latter has zero indeterminacy then
$$\langle\alpha_1,\beta_1,\gamma_1\rangle\times\alpha_2\beta_2\gamma_2=\pm\langle\alpha_1\times\alpha_2,\beta_1\times\beta_2,\gamma_1\times\gamma_2\rangle.$$
(The similar relation
$\alpha_1\beta_1\gamma_1\times\langle\alpha_2,\beta_2,\gamma_2\rangle=\pm\langle\alpha_1\times\alpha_2,\beta_1\times\beta_2,\gamma_1\times\gamma_2\rangle$
holds when $\langle\alpha_2,\beta_2,\gamma_2\rangle\subseteq H^*(X_2)$ is defined and
$\alpha_1,\beta_1,\gamma_1\in H^*(X_1)$ are arbitrary.)
\end{prop}
\begin{proof}
Properties {\bf (a), (b)} and {\bf (c)} follow immediately from the
definition. Property {\bf (d)} is Corollary 7 of Kraines [Kra]
(there higher Massey products are treated, of which the triple product
is a special case). Property {\bf (e)} follows from properties {\bf (c)} and {\bf (d)} together with the identities
$$(\alpha_1\times\alpha_2)\cup(\beta_1\times\beta_2)=(-1)^{|\alpha_2||\beta_1|}(\alpha_1\cup\beta_1)\times(\alpha_2\cup\beta_2),$$
$$\alpha_1\times\alpha_2=p_1^*(\alpha_1)\cup p_2^*(\alpha_2),$$
for all $\alpha_1,\beta_1\in H^*(X_1)$, $\alpha_2,\beta_2\in
H^*(X_2)$, where $p_i\co X_1\times X_2\to X_i$ is projection onto
$X_i$ for $i=1,2$.
\end{proof}
\begin{prop}\label{external}
Let $R=\mathbf{F}$ a field, and let $X_1$, $X_2$ be spaces of finite
type. Suppose that the Massey product
$$\theta=\langle\alpha_1\times\alpha_2,\beta_1\times\beta_2,\gamma_1\times\gamma_2\rangle\subseteq
H^*(X_1\times X_2;\mathbf{F})\cong H^*(X_1;\mathbf{F})\otimes
H^*(X_2;\mathbf{F})$$ is defined. If either
$\alpha_1\beta_1=\beta_2\gamma_2=0$ or
$\alpha_2\beta_2=\beta_1\gamma_1=0$ then $\theta$ contains the zero class.
\end{prop}

\begin{proof}
The Eilenberg-Zilber
Theorem gives a chain equivalence
$$\mathrm{EZ}\co C^*(X_1\times X_2)\to C^*(X_1)\otimes C^*(X_2)$$
which can be seen to be a mapping of DGA's (the
product and differential on the right hand side are given respectively by
$$(a\otimes b)(c\otimes d)=(-1)^{|b||c|}a\bullet c\otimes b\bullet
d,\quad d_{\otimes}(a\otimes b)=da\otimes b +(-1)^{|a|}a\otimes db,$$
where $\bullet$ denotes usual cochain multiplication in
$C^*(X_i)$). Hence we may compute Massey products in $H^*(X_1\times
X_2)$ using the cochain complex $(C^*(X_1)\otimes
C^*(X_2),d_{\otimes})$, and we find that $0\in \theta$ if and only
if
\begin{equation}
0\in \langle [a_1\otimes a_2],[b_1\otimes b_2],[c_1\otimes c_2]\rangle,
\end{equation}
where the $a_i$, $b_i$ and $c_i$ are cocycles
representing $\alpha_i$, $\beta_i$ and $\gamma_i$.

Suppose that $\alpha_1\beta_1=\beta_2\gamma_2=0$. Let $\mu'\in
C^*(X_1)$ and $\lambda'\in C^*(X_2)$ be cochains such that
$d\mu'=(-1)^{|b_1||a_2|}a_1\bullet b_1$ and
$d\lambda'=(-1)^hb_2\bullet c_2$, where $h=|c_1|(|b_2|-1)-|b_1|$. One may show
that the cochains $\mu=\mu'\otimes a_2\bullet b_2$ and
$\lambda=b_1\bullet c_1\otimes\lambda'$ satisfy
$$d_{\otimes}\mu =(a_1\otimes a_2)(b_1\otimes b_2),\quad
d_{\otimes}\lambda =(b_1\otimes b_2)(c_1\otimes c_2).$$
Hence the above Massey product (5) contains the class represented by the
cocycle
$$(a_1\otimes a_2)\lambda +(-1)^{|a_1|+|a_2|+1}\mu\otimes (c_1\otimes
c_2).$$
A quick calculation gives that this cocycle is the coboundary
$$d_{\otimes}((-1)^{|c_1||a_2|}\mu'\bullet c_1\otimes a_2\bullet
  \lambda')$$
and hence represents zero.

The proof that $\theta$ contains zero when
$\alpha_2\beta_2=\beta_1\gamma_1=0$ runs similarly.
\end{proof}

\section{Examples.}

We now present examples of non-formal spaces $X$ where non-zero Massey
products in $H^*(X)$ allow us to apply the results of previous
Sections to obtain better lower bounds for $\TC(X)$ than the
zero-divisors cup-length. In all our examples we consider cohomology
with coefficients in the field $\mathbf{Q}$ of rational numbers. If $u\in H^\ell(X)$ is a cohomology class, it will be convenient to
denote by $\overline{u}$ the class
$$\overline{u}=1\times u-u\times 1\in H^\ell(X\times X).$$
\begin{exam}\label{Whitehead}{\em 
Let $X=S^3_a\vee S^3_b\cup e^8\cup e^8$ be the space obtained from the
wedge of two copies of the 3-sphere by attaching 8-cells by means of
the iterated Whitehead products $[S^3_a,[S^3_a,S^3_b]]$ and
$[S^3_b,[S^3_a,S^3_b]]$. This is one of the simplest examples of a
simply-connected non-formal space. We will show that $\TC(X)=5$, while the zero-divisors cup-length is 2.

First we note that since $X$ is a $2$-connected, $8$-dimensional
CW-complex, Proposition 5.1 of [Jam] gives
$\cat(X)<\frac{8+1}{2+1}+1=4.$
Therefore $\cat(X)\leq 3$ (in fact $\cat(X)=3$; see below) and the
upper bound $\TC(X)\leq 2\cdot\cat(X)-1$ given by Theorem 5 of [Far1] gives
$\TC(X)\leq 5$.

Let $a,b\in H^3(X)$ be the generators corresponding to the
two spheres. It is known ([UM], Lemma 7) that the Massey products $\langle a,a,b\rangle$ and $\langle
b,a,b\rangle$ are non-zero linearly independent elements of $H^8(X)$
(the indeterminacy is zero, since cup products are trivial in $H^*(X)$ for dimensional reasons). Since
$\wgt(\langle a,a,b \rangle)\geq 2$ by [Rud] Theorem 4.6, we can
apply our Theorem \ref{main} with $r=k=2$ to conclude that
$\wgt_{\pi_X}(\overline{\langle a,a,b \rangle})\geq 2$. Similarly
$\wgt_{\pi_X}(\overline{\langle b,a,b \rangle})\geq 2$. Now since
$$ \overline{\langle a,a,b \rangle}\cdot\overline{\langle b,a,b \rangle}=-\langle a,a,b \rangle\times\langle b,a,b \rangle -\langle b,a,b \rangle\times\langle a,a,b \rangle\neq 0,$$
Theorem \ref{weight} gives $\TC(X)>4$, so $\TC(X)=5$.
}
\end{exam}
\begin{rem}{\em
 Example \ref{Whitehead} is also considered in paper [FGKV], where the authors construct an invariant $\mathbf{MTC}(X)$ which is a lower bound for $\TC(X)$ using an explicit semi-free model of the fibred join (see Example 6.7 there). One suspects that the results there are related to ours. The methods here appear to give stronger lower bounds, as well as being simpler and more widely applicable; for instance we may also treat non-simply-connected spaces, as in the next Example.}
 \end{rem}
\begin{exam}{\em
 Let $X=S^3-B$ be the link complement of the Borromean rings. In his seminal paper [Mas] Massey gave a rigorous proof that the Borromean rings link is not isotopic to the unlink, by exhibiting non-zero triple products in $H^*(X)$. By Alexander duality we have $H^1(X)=\mathbf{Q}^3$ and $H^2(X)=\mathbf{Q}^2$. The generators $u,v,w\in H^1(X)$ are represented by cocycles dual to the disks
spanned by each of the embedded circles. The cup product
structure in $H^*(X)$ is trivial, reflecting algebraically the fact that the linking number of each pair of circles is zero. However, the Massey
products $\langle u,v,w\rangle$ and $\langle u,w,v\rangle$ are
non-zero linearly independent elements of $H^2(X)$ ([Mas], Theorem 3.1).

We claim that the Massey product $\theta=\langle
\overline{u},-\overline{v}\overline{\langle
  u,w,v\rangle},\overline{w}\rangle$ of degree 4 is non-zero in $H^*(X\times X)=H^*(X)\otimes H^*(X)$. In fact
\begin{eqnarray*}
\theta &=&\langle 1\otimes u-u\otimes 1,\langle u,w,v\rangle\otimes v+v\otimes\langle u,w,v\rangle,1\otimes w-w\otimes 1\rangle \\
 &\subseteq &\langle 1\otimes u,\langle u,w,v\rangle\otimes v,1\otimes w\rangle -\langle u\otimes 1,\langle u,w,v\rangle\otimes v,1\otimes w\rangle \\
 & &+\langle u\otimes 1,v\otimes \langle u,w,v\rangle ,w\otimes 1\rangle - \langle 1\otimes u,v\otimes \langle u,w,v\rangle ,w\otimes 1\rangle\\
  & &+\langle 1\otimes u,v\otimes \langle u,w,v\rangle ,1\otimes w\rangle -\langle u\otimes 1,v\otimes \langle u,w,v\rangle ,1\otimes w\rangle\\
  & &+\langle u\otimes 1,\langle u,w,v\rangle\otimes v,w\otimes 1\rangle -\langle 1\otimes u,\langle u,w,v\rangle\otimes v,w\otimes 1\rangle\\
  &=&\pm\langle u,v,w\rangle\otimes\langle u,w,v\rangle\pm\langle u,w,v\rangle\otimes\langle u,v,w \rangle\neq 0.
  \end{eqnarray*}
  The inclusion follows from Proposition \ref{multi} {\bf (a)} and {\bf (b)}. To
  obtain the second equality, first observe that any Massey product
  $\langle\alpha,\beta,\gamma\rangle$ in $H^*(X)\otimes H^*(X)$ with
  $|\alpha|=|\gamma|=1$ and $|\beta|=3$ has indeterminacy zero. Those
  Massey products with positive sign now sum to give the right-hand side, by
  Proposition \ref{multi} {\bf (e)}. Those with negative sign are zero by
  Proposition \ref{external}.

Therefore Theorem \ref{rudyak} gives $\TC(X)>3$, while the zero-divisors
cup-length equals 2. Combined with the upper bound $\TC(X)\leq 2\cdot\cat(X)-1$ this
gives $\TC(X)=4$ or $5$. This is the first known example of an aspherical space for which $\TC$ is greater than zero-divisors cup-length plus one.}
\end{exam}
\begin{questions}
Is $\wgt_{\pi_X}(\overline{\langle
  u,w,v\rangle})=2$? (If so then Theorem \ref{rudyak} gives $\TC(X)>4$, and hence $\TC(X)=5$.) Is there a result analogous to Theorem \ref{main} for non-simply-connected spaces?
  \end{questions}
  \begin{problem}{\em
  Give an expression for the topological complexity of a given knot or link complement in terms of known invariants.}
  \end{problem}


\begin{exam}{\em
 Let $\xi$ be a $2m$-dimensional vector bundle over
  $S^{m}\times S^m$ whose Euler class $e(\xi)\in H^{2m}(S^m\times S^m)$ is
  non-zero (here $m\geq 2$). Let $X$ denote the total
  space of the unit sphere bundle of $\xi$ over
  $S^{m}\times S^{m}$. We will show that $\TC(X)\geq 6$ while the zero-divisors cup-length is 3.

If $m$ is even, the Sullivan minimal model for $X$ has the form $(\Lambda\{a,b,x,y,z\},d)$
where
\[
da=db=0,\quad dx=a^2,\quad dy=b^2,\quad dz=ab,
\]
and $|a|=|b|=m$ and $|x|=|y|=|z|=2m-1$. A basis for $H^*(X)$ is
therefore given by the elements $\alpha=[a],\beta=[b]\in H^{m}(X)$,
$u=[az-xb],v=[bz-ya]\in H^{3m-1}(X)$ and $\mu=[abz-ya^2]\in
H^{4m-1}(X)$. The only non-trivial cup-products are $\alpha
v=\mu=u\beta$.

If $m$ is odd the minimal model has the form $(\Lambda\{a,b,z\},d)$ where
\[
da=db=0,\quad dz=ab,
\]
and $|a|=|b|=m$ and $|z|=2m-1$. A basis for $H^*(X)$ is
given by the elements $\alpha=[a],\beta=[b]\in H^{m}(X)$,
$u=[az],v=[zb]\in H^{3m-1}(X)$ and $\mu=[azb]\in
H^{4m-1}(X)$, and again the only non-trivial products are $\alpha
v=\mu=u\beta$.

In both cases $u\in\langle\alpha,\alpha,\beta\rangle$ and
$v\in\langle\beta,\beta,\alpha\rangle$, and hence Theorem 4.6 of [Rud] gives
$\wgt(u)\geq 2$ and $\wgt(v)\geq 2$. Our Theorem \ref{main} now applies with
$k=2$, $r=m-1$ to give $\wgt_{\pi_X}(\overline{u})\geq 2$ and
$\wgt_{\pi_X}(\overline{v})\geq 2$. Now since
$$\overline{\alpha}\cdot\overline{u}\cdot\overline{v}=-u\times\mu\pm\mu\times
u\neq 0,$$
and $\wgt_{\pi_X}(\overline{\alpha})\geq 1$ as $\overline{\alpha}$ is
a zero-divisor, Theorem \ref{weight} gives $\TC(X)> 1+2+2=5$.

In general $\TC(X)=6$ or $7$, since $\cat(X)=4$ (see Example 4.9 of
[Rud]; recall that our definition of category differs from that of
[Rud] by one).}
\end{exam}


\begin{thebibliography}{[FGKV]}

\bibitem[FH]{[FH]} E. Fadell and S. F. Husseini, {\it Category weight and
  Steenrod operations}\/,  Bol.\ Soc.\ Mat.\ Mexicana (2)  37
  (1992), no. 1-2, 151--161.

\bibitem[Far1]{[Far1]} M. Farber,
{\it Topological complexity of motion planning}\/,
 Discrete Comput.\ Geom.  29 (2003), 211--221.

\bibitem[Far2]{[Far2]} M. Farber,
{\it Instabilities of robot motion}\/,
 Topology Appl.  140 (2004), 245--266.

\bibitem[Far3]{[Far3]} M. Farber,
{\it Topology of robot motion planning},
 in: Morse Theoretic Methods in Nonlinear Analysis and in Symplectic Topology (P.\ Biran et al (eds.))  (2006), 185--230.

\bibitem[FG1]{[FG1]} M. Farber and M. Grant, {\it Symmetric Motion
  Planning}\/, in: Topology and Robotics (M.\ Farber et al (eds.)),
Contemporary Mathematics 438 (2007), 85--104.

\bibitem[FG2]{[FG2]} M. Farber and M. Grant, {\it Robot motion planning,
  weights of cohomology classes, and cohomology operations}\/,  to
    appear in Proc.\ Amer.\ Math.\ Soc. Available online at http://arxiv.org/abs/0706.2497

\bibitem[FTY]{[FTY]} M. Farber, S. Tabachnikov, and S. Yuzvinsky, {\it Topological robotics: motion planning in
  projective spaces}\/,  Int. Math. Res. Not.  2003,  no. 34,
  1853--1870.

\bibitem[FGKV]{[FGKV]} L. Fern\'{a}ndez Su\'{a}rez, P. Ghienne, T. Kahl
    and L. Vandembroucq, {\it Joins of DGA modules and sectional
  category}\/,  Algebr.\ Geom.\ Topol.  6 (2006), 119--144.

\bibitem[Jam]{[Jam]} I. M. James, {\it On category, in the sense of
    Lusternik-Schnirelmann}\/, Topology 17 (1978), 331--349.

\bibitem[Kra]{[Kra]} D. Kraines, {\it Massey higher products}\/,  Trans.\ Amer.\ Math.\ Soc.  124 (1966), 431--449.

\bibitem[Mas]{[Mas]} W. S. Massey, {\it Higher order linking numbers},
   in: Conference on Algebraic Topology (Chicago, 1968), 174--205.

\bibitem[May]{[May]} J. P. May, {\it Matric Massey products}\/,  J.\
    Algebra  12 (1969), 533--568.

\bibitem[Mil]{[Mil]} J. W. Milnor, {\it Construction of universal bundles. II}\/,   Ann. of Math. (2)  63  (1956), 430--436.

\bibitem[Rud]{[Rud]} Y. B. Rudyak, {\it On category weight and its
  applications}\/,  Topology  38 (1999),  no. 1, 37--55.

\bibitem[Sch]{[Sch]} A. S. Schwarz, {\it The genus of a fiber space}\/,
 Amer.\ Math.\ Soc.\ Transl.(2)  55 (1966), 49--140.

\bibitem[Sma]{[Sma]} S. Smale, {\it On the topology of algorithms I}\/,
   J.\ Complexity  3 (1987), 81--89.

\bibitem[Str]{[Str]} J. Strom, {\it Category weight and essential category
    weight}\/, Ph.\ D.\ Thesis, University of Wisconsin-Madison
  (2002).

\bibitem[UM]{[UM]} H. Uehara and W. S. Massey, {\it The Jacobi identity
  for Whitehead products}\/,
  in: Algebraic Geometry and Topology, A
    Symposium in Honor of S.\ Lefschetz (1957), 361--377.

\bibitem[Vas1]{[Vas1]} V. A. Vassiliev, {\it Cohomology of braid groups and complexity of algorithms}\/,  Functional Anal.\ Appl.  22
(1988) 15--24.

\bibitem[Vas2]{[Vas2]} V. A. Vassiliev, {\it Topology of Complements to
  Discriminants}\/, FAZIS, Moscow, 1997.

\bibitem[Whi]{[Whi]} G. W. Whitehead,  Elements of homotopy theory,
  Springer-Verlag, 1979.

\end{thebibliography}
\end{document}